\documentclass[12pt]{amsart}

\usepackage[left=1in,right=1in,top=1in,bottom=1in]{geometry} 

\usepackage{amsmath}

\usepackage{amssymb,amsthm,amsmath}

\usepackage{ifthen}

\usepackage{graphicx,float}
\usepackage{caption,subcaption}
\usepackage{enumerate}

\DeclareGraphicsExtensions{.eps}

\newtheorem{theorem}{Theorem}[section]

\newtheorem{lemma}[theorem]{Lemma}

\newtheorem{proposition}[theorem]{Proposition}

\newtheorem{corollary}[theorem]{Corollary}

\theoremstyle{definition}

\newtheorem{definition}[theorem]{Definition}

\newtheorem{remark}[theorem]{Remark}

\numberwithin{equation}{section}

\newcommand{\C}{\mathbb{C}}

\newcommand{\e}{\epsilon}

\newcommand{\HI}{\mathbb{H}}

\newcommand{\BL}{\mathcal{B}}

\newcommand{\CL}{\mathcal{C}}

\makeatletter
\@namedef{subjclassname@2020}{\textup{2020} Mathematics Subject Classification}
\makeatother

\begin{document}

\baselineskip=21pt
\title[Pseudospectra of the heat operator pencil]
{Pseudospectra of the heat operator pencil}

\author{Krishna Kumar G.}
\address{Department of Mathematics, University of Kerala, Kariavattom Campus, Thiruvananthapuram, Kerala, India, 695581.}
{\email{krishna.math@keralauniversity.ac.in}
\thanks{The first author is supported through the SERB MATRICS grant with project reference no. MTR/2021/000028 and the PLEASE scheme, Department of Higher Education, Government of Kerala, Kerala, India.}
\author{Judy Augustine}
\address{Department of Mathematics, University of Kerala, Kariavattom Campus, Thiruvananthapuram, Kerala, India, 695581.}
 {\email{judyaugustine9@gmail.com}

\thanks{The second author thanks Kerala State Council for Science, Technology and Environment (KSCSTE), Kerala, India (Ref No: KSCSTE/990/2021-FSHP-MS) for financial support.}  
\subjclass[2020]{Primary ; Secondary } 


\begin{abstract}
This article undertakes an analysis of the one-dimensional heat equation, wherein the Dirichlet condition is applied at the left end and Neumann condition at the right end. The heat equation is restructured as a non-self-adjoint $2\times 2$ unbounded block operator matrix pencil. The spectral, pseudospectral, and $(n,\e)$-pseudospectral enclosures of the $2\times 2$ unbounded block operator matrix pencil are explored to scrutinize the heat operator pencil. The plots of the discretized equation are depicted to illustrate the observations.
\end{abstract}
\maketitle
\section{Introduction}
The heat transfer through an ideal rod of length $d$ with thermal diffusivity $c^2$ subject to the Dirichlet condition at the left and Neumann condition at the right end is formulated as,
\begin{align}\label{heat1}
\frac{\partial \phi}{\partial t}= c^2\frac{\partial^2 \phi}{\partial x^2}, \ \ 0\leq x\leq d, \ \  t\geq 0, \ \  \phi(0,t)= \phi_x(d,t)= 0.
\end{align}
The domain of the solution (one-dimensional heat equation) is a semi-infinite strip of width $d$ that continues indefinitely in time. The probabilistic approach, variational iteration, finite difference, and finite element are among the few methods used to study the one-dimensional heat equation; see \cite{doob, rama, sha, zlam}. The one-dimensional heat equation is reformulated in this article as a linear evolution process on a suitable Hilbert space. The closed linear operator pencil arising from the heat equation is illustrated as follows. Let $u= \begin{pmatrix} u_1 \\u_2 \end{pmatrix}= \begin{pmatrix}
\phi\\ \phi_x \end{pmatrix}$, then the system (\ref{heat1}) is converted into 
\begin{align}\label{heat2}
\mathcal{B} u_t= \mathcal{A} u, \ \ \mathcal{A}= \begin{pmatrix} c^2\cfrac{\partial^2}{\partial x^2} &&\\ && \end{pmatrix}, \ \ \mathcal{B}= \begin{pmatrix}I&\\ \cfrac{\partial}{\partial x} &-I\end{pmatrix}, \ \ u_1(0,t)= u_2(d,t)=0.
\end{align}
Suppose $u$ is variably separable and let 
\[
u(x,t)= e^{\lambda t}\psi(x)= e^{\lambda t} \begin{pmatrix}
\psi_1(x)\\\psi_2(x)
\end{pmatrix}.
\]
Then the system (\ref{heat2}) is converted into the generalized eigenvalue problem of the form 
\begin{align}\label{gevp}
(\lambda B-A)\psi= 0, \ \ A= \begin{pmatrix} c^2\cfrac{d^2}{dx^2} &&\\ && \end{pmatrix}, \ \ B= \begin{pmatrix}I&\\ \cfrac{d}{dx} &-I\end{pmatrix}, \ \ \psi_1(0)= \psi_2(d)= 0.
\end{align}
Thus, the one-dimensional heat equation defined in (\ref{heat1}) is transformed into a linear operator pencil, a $2\times 2$ unbounded block operator matrix pencil. If $(\lambda B-A)\psi= 0$ for some $\psi\neq 0$, then $\lambda$ is called the generalized eigenvalue and $\psi$ is the corresponding generalized eigenfunction. The domains $D(A)$ and $D(B)$ are suitable subspaces of the Hilbert space $L^2[0,d]\times L^2[0,d]$ defined by
\begin{align*}
D(A)&= \left\{v\in L^2[0,d]: v' \ \textnormal{is absolutely continous}, \ v''\in L^2[0,d]\right\}\times L^2[0,d],
\end{align*}
and 
\begin{align*}
D(B)&= \left\{v \in L^2[0,d]: v \ \textnormal{is absolutely continous}, \ v'\in L^2[0,d]\right\}\times L^2[0,d].
\end{align*}
Note that $B^2=I$ on $D(B)$ and for $\psi= \begin{pmatrix}\psi_1\\ \psi_2\end{pmatrix}\in L^2[0,d]\times L^2[0,d]$,
\begin{align*}
\|\psi\|^2= \int_0^{d} \left(\left|\psi_1(x)\right|^2+\left|\psi_2(x)\right|^2\right) \ dx.
\end{align*}
In \cite{nag}, Nagel introduces the matrix theory for unbounded operator matrices. One can find vast information about block operator matrices in \cite{tret}. This article focuses on the pseudospectral study of closed linear operator pencil and $2\times 2$ unbounded block operator matrix pencil to analyze the one-dimensional heat equation. 
\subsection{Background and Outline}
The operators that arise from the physical applications are closed but unbounded, and due to this, closed operators gained attention in the mathematical field. For more about closed operators, refer \cite{kato, kre}. Throughout this article, $\HI$ denotes a complex Hilbert space, and $I, {\bf 0}$ are the identity operator, zero operators on $\HI$. Further $\CL(\HI), \BL(\HI)$ are the set of all closed, bounded operators on $\HI$. The domain of $A\in \CL(\HI)$ is denoted as $D(A)$.

Let $A, B \in \CL(\HI)$ and $\lambda\in \C$, then the closed linear operator pencil 
\[
(A,B)(\lambda)= \lambda B-A
\]
is defined on $D(A)\cap D(B)$. If both $A$ and $B$ are self-adjoint, then $(A, B)$ is called a self-adjoint pencil. For more information on self-adjoint operator pencils, see \cite{mark}. The generalized resolvent of $(A,B)$ is defined by 
\[
\rho(A,B)=  \{\lambda \in \mathbb{C}:  (\lambda B-A)^{-1}\in \BL(\HI)\}.
\]
The spectrum of $(A, B)$ is defined by $\sigma(A, B)= \mathbb{C}\setminus \rho(A, B)$. The generalized eigenvalues of $(A,B)$ is defined by
\[
\sigma_e(A,B)= \{\lambda\in \C: \lambda B-A \ \textnormal{is not one-one}\}.
\]
The spectral analysis of operators fails to provide the correct information while dealing with non-normal operators. During the 1990s, Trefethen advocated the concept of pseudospectrum in the context of non-normal operators resulting from the discretization of certain non-symmetric differential operators; see \cite{tre1,tre2}.
\begin{definition} 
Let $A \in \CL(\HI)$ and $\e>0$, the $\e$-pseudospectrum of $A$ is defined by
\[
\Lambda_{\epsilon}(A)= \sigma(A) \cup \left\{\lambda \in \C: \left\|(\lambda I-A)^{-1}\right\| \geq \e^{-1}\right\}. 
\]
\end{definition}
Pseudospectra has been extensively studied due to its applications in numerical analysis and differential equations. It is used to study the norm behavior of $A^n$ ($n= 1,2,\ldots$) and $e^{tA}$ ($t>0$), analyze iterative methods for the solutions of $Ax=b$, and in stability analysis of the method of lines discretizations of time-dependent partial differential equations; see \cite{reddy, tre1, tre3}. Over the periods, pseudospectrum has been studied for various differential operators like wave, convection-diffusion, and Orr-Sommerfield operators; see \cite{dri, scrc, scro}. Pseudospectra provides more stable information about non-self-adjoint operators under various limiting procedures than the spectrum; see \cite{davis}. Later on, Hansen introduced the more general $(n,\e)$-pseudospectra for linear operators on separable Hilbert spaces and pointed out that they have numerous attractive features with pseudospectra but offer a better insight into the approximation of the spectrum; see \cite{han1,han2}. The $(n,\e)$-pseudospectrum of bounded linear operator pencils is studied in \cite{kk}.  

The article is organized as follows. Section 1 illustrates the formulation of the one-dimensional heat equation into a closed linear operator pencil. The suitable substitution and the separation of variables transform the one-dimensional heat equation to a $2\times 2$ unbounded block operator matrix pencil. Section 2 studies the $(n,\e)$-pseudospectra of closed linear operator pencil. Several important properties of the $(n,\e)$-pseudospectrum of a closed linear operator pencil are developed. The equivalent definitions for the $(n,\e)$-pseudospectra of closed linear operator pencil and the characterization of the $(n,\e)$-pseudospectra of self-adjoint operator pencil are done. Section 3 concerns the  $2\times 2$ unbounded block operator matrix pencil. Using the generalized Frobenius-Schur factorization, we furnish the spectral, pseudospectral, and $(n,\e)$-pseudospectral enclosures for $ 2\times 2$ unbounded block operator matrix pencil. Section 4 identifies the heat operator pencil as a non-self-adjoint pencil. The eigenvalue analysis and the pseudospectral enclosure of the heat operator pencil are done. The pseudospectrum of discretized heat equation is plotted to illustrate the results.
\section{$(n,\e)$-pseudopsectra of the closed linear operator pencil}
This section studies the $(n,\e)$-pseudospectrum of a closed linear operator pencil for a general purpose. Let $\mathbb{Z}_{+}$ denote the set of all positive integers. For $z\in \C$ and $r>0$, $D(z,r):= \{\lambda: |\lambda-z|\leq r\}$ and $D(0,r)$ is denoted as $\Delta_r$. For $\Omega\subseteq \C$, $\overline{\Omega}:= \{\overline{\lambda}: \lambda\in \Omega\}$.
\begin{definition}
Let $A,B \in \CL(\HI)$, $n\in \mathbb{Z}_{+}\cup \{0\}$, and $\e >0$. The $(n,\e)$-pseudospectrum of $(A,B)$ is denoted by $\Lambda_{n,\epsilon}(A,B)$ 
and is defined by
\[ 
 \Lambda_{n,\epsilon}(A,B)= \sigma(A,B) \cup \left\{\lambda \in \C: \left\|(\lambda B-A)^{-2^n} \right\|^{\frac{1}{2^n}} \geq \e^{-1}\right\}.
 \]
The generalized $(n,\epsilon)$-pseudoresolvent of $(A,B)$ is denoted by $\rho_{n,\epsilon}(A,B)$ and is defined by 
\[
 \rho_{n,\epsilon}(A,B)= \rho (A,B) \cap \left\{\lambda \in \C: \left\|(\lambda B -A)^{-2^n} \right\|^{\frac{1}{2^n}} <  \e^{-1}\right\}.
\]
\end{definition}
The following observations about the pseudospectra of closed linear operator pencil are trivial.
\begin{remark}
Let $A,B \in \CL(\HI)$, $n\in \mathbb{Z}_{+}\cup \{0\}$, and $\e >0$. Then
\begin{enumerate}
\item $\sigma (A,B)\subseteq \Lambda_{n,\epsilon}(A,B)$.
\item If $ B=I$, then $\Lambda_{n,\epsilon}(A,I)=  \Lambda_{n,\epsilon}(A)$.
\item If $n=0$, then $\Lambda_{n,\epsilon}(A,B)=  \Lambda_{\epsilon}(A,B)$.
\end{enumerate}
\end{remark}
\begin{lemma}\label{lem1}
Let $A, B\in \CL(\HI)$ and $\e>0$. Then $\sigma(A+C,B)\subseteq \Lambda_{\e}(A,B)$ for every $C\in \BL(\HI)$ with $\|C\|\leq \e$.
\end{lemma}
\begin{proof}
Suppose $\lambda\in \sigma(A+C,B)$ and $\lambda \notin \sigma(A,B)$, then 
\[\lambda B- A-C= \left(\lambda B- A\right)\left(I-\left(\lambda B- A\right)^{-1}C \right).
\]
It follows that 
\[
1\leq \left\|\left(\lambda B- A\right)^{-1}C  \right\| \leq \left\|\left(\lambda B- A\right)^{-1}\right\| \left\|C \right\|.
\]
Hence $\lambda\in \Lambda_\e(A,B)$ for $\|C\|\leq \e$.
\end{proof}
\begin{theorem}\label{2}
Let $A, B\in \CL(\HI)$, $n \in \mathbb{Z}_{+} \cup \{0\}$, and $\e> 0$. Then the following holds.
\begin{enumerate}
\item $\Lambda_{n+1,\epsilon}(A,B) \subseteq  \Lambda_{n,\epsilon}(A,B)$.
\item $\Lambda_{n,\epsilon_1}(A,B) \subseteq  \Lambda_{n,\epsilon_2}(A,B)$  for every $0< \epsilon_1\leq \epsilon_2$.
\item $\displaystyle \sigma(A,B)= \bigcap_{\epsilon >0}\Lambda_{n,\epsilon}(A,B)$.
\item $\Lambda_{n,\e}(A,B)+ \Delta_\delta \subseteq \Lambda_\e(A,B)+\Delta_\delta\subseteq \Lambda_{\e+\delta\|B\|}(A,B)$ for every $B\in \BL(\HI)$ and $\delta>0$.
\item $\Lambda_{n,\epsilon} ( \alpha A,\alpha B)=  \Lambda_{n,\frac{\epsilon}{|\alpha|}} (A,B) $ for every $\alpha \neq 0$.
\item $\Lambda_{n,\epsilon}(\beta A+\alpha B, B)= \alpha+ \beta \, \Lambda_{n,\frac{\epsilon}{|\beta|}}(A,B)$ for every $\alpha, \beta \in \C$ with $\beta\neq 0.$
\end{enumerate}
\end{theorem}
\begin{proof}
\begin{enumerate}
\item Suppose $\lambda\in \Lambda_{n+1,\epsilon}(A,B)\setminus \sigma(A,B)$. Then
\begin{align*}
\e^{-1}\leq \left\|[(\lambda B-A)^{-1}]^{2^{n+1}}\right\|^{\frac{1}{2^{n+1}}} 
&\leq \left\|[(\lambda B-A)^{-1}]^{2^n} \right\|^{\frac{1}{2^{n+1}}} \left\|[(\lambda B-A)^{-1}]^{2^n} \right\|^{\frac{1}{2^{n+1}}}\\
&= \left\|[(\lambda B-A)^{-1}]^{2^n} \right\|^{\frac{1}{2^n}}. 
\end{align*}
Hence $\lambda\in \Lambda_{n,\epsilon}(A,B)\setminus \sigma(A,B)$.
\item Let $0<\e_1\leq \e_2$ and $\lambda \in \Lambda_{n,\epsilon_1}(A,B)$. If $\lambda\in \sigma(A,B)$, then $\lambda\in \Lambda_{n,\epsilon_2}(A,B)$. If $\lambda \in \Lambda_{n,\epsilon_1}(A,B)\setminus \sigma(A,B)$, then
\[
 \left\|(\lambda B-A)^{-2^n} \right\|^{\frac{1}{2^n}} \geq \epsilon_1^{-1} \geq\epsilon_2^{-1}. 
\]
Hence $\lambda \in \Lambda_{n,\epsilon_2}(A,B)$.
\item If $\lambda\in \sigma(A,B)$, then $\lambda\in \Lambda_{n,\e}(A,B)$ for every $\e>0$ and $\displaystyle \lambda\in \bigcap_{\e>0}\Lambda_{n,\e}(A,B)$. Next assume that $\displaystyle \lambda\in \bigcap_{\e>0}\Lambda_{n,\e}(A,B)$ and $\lambda\notin \sigma(A,B)$, then 
\[
\left\|(\lambda B -A)^{-2^n} \right\|^{\frac{1}{2^n}} \geq \epsilon^{-1} \ \ \ \  \text{for every} \ \epsilon > 0.
\]
Hence $\left\|(\lambda B-A)^{-2^n}\right\|= 0$, a contradiction.
\item From (1), $\Lambda_{n,\e}(A,B) \subseteq\Lambda_{\e}(A,B)$ and hence $\Lambda_{n,\e}(A,B)+\Delta_{\delta} \subseteq\Lambda_{\e}(A,B)+\Delta_{\delta}$. Let $\lambda\in \Lambda_{\e}(A,B)$, $\mu\in \C$, and $|\mu|\leq \delta$. If $\lambda\in \sigma(A,B)$, from Lemma \ref{lem1},
\[
\lambda+\mu\in \sigma(A+\mu B,B)\subseteq \Lambda_{\delta\|B\|}(A,B) \subseteq \Lambda_{\e+ \delta\|B\|}(A,B).
\]
If $\lambda\in \Lambda_{\e}(A,B)\setminus \sigma(A,B)$ and $\lambda+\mu\in \sigma(A,B)$, then $\lambda+\mu\in \Lambda_{\e+\delta\|B\|}(A,B)$. Further if $\lambda\in \Lambda_{\e}(A,B)\setminus \sigma(A,B)$ and $(\lambda+\mu)B-A$ is invertible. Choose $x\in D(A)$ with $\|x\|= 1$ such that $\|(\lambda B-A)x\|\leq \e$, then
\begin{align*}
1=\|x\|&=\left\| \left(\left(\lambda+\mu\right) B- A\right)^{-1} \left(\left(\lambda+\mu\right) B- A\right)x\right\|\\
&\leq \left(\e+ \delta\|B\|\right)\left\| \left(\left(\lambda+\mu\right) B- A\right)^{-1}\right\|.
\end{align*}
Hence $\lambda+\mu\in \Lambda_{\e+ \delta\|B\|}(A,B)$.
\item Let $\alpha\neq 0$, then $\sigma(\alpha A,\alpha B)= \sigma(A,B)$. Also
\begin{align*}
\lambda \in \Lambda_{n,\epsilon}(\alpha A,\alpha B)\setminus \sigma(\alpha A, \alpha B)& \Longleftrightarrow \left\|(\lambda \alpha B - \alpha A)^{-2^n} \right\|^{\frac{1}{2^n}} \geq \e^{-1}\\
 & \Longleftrightarrow \left\|(\lambda B- A)^{-2^n} \right\|^{\frac{1}{2^n}} \geq \frac{|\alpha|}{\epsilon}\\
  & \Longleftrightarrow \lambda \in \Lambda_{n,\frac{\epsilon}{|\alpha|}} (A,B) \setminus \sigma(A,B). 
\end{align*}
 \item Let $\alpha,\beta\in \C $ and $\beta\neq 0$, then 
 \begin{align*}
  \rho (\beta A+\alpha B,B) &= \{\lambda \in \mathbb{C}:\left(\lambda B-\beta A -\alpha B \right)^{-1}\in \BL(\HI)\}\\
  &= \left\{\lambda \in \mathbb{C}:\left( \frac{\lambda-\alpha}{\beta}B -A\right)^{-1} \in \BL(\HI)\right\}.
 \end{align*}
 i.e., $\lambda\in \sigma(\beta A+\alpha B,B) \Longleftrightarrow \frac{\lambda-\alpha}{\beta} \in \sigma(A,B)
       \Longleftrightarrow  \lambda \in \alpha+ \beta \, \sigma(A,B).$ Also
 \begin{align*}
\lambda \in \Lambda_{n,\epsilon}(\beta A + \alpha B,B)\setminus \sigma(\beta A+\alpha B,B) & \Longleftrightarrow \left\|(\lambda B-\beta A -\alpha B)^{-2^n}\right\|^{\frac{1}{2^n}} \geq \e^{-1}\\
& \Longleftrightarrow \ \frac{1}{|\beta|} \left\|\left(\frac{\lambda -\alpha}{\beta} B -A\right)^{-2^n}\right\|^{\frac{1}{2^n}} \geq \e^{-1}\\
& \Longleftrightarrow \  \frac{\lambda -\alpha}{\beta} \in \Lambda_{n,\frac{\epsilon}{|\beta|}} (A,B) \setminus \sigma(A,B)\\
& \Longleftrightarrow \  \lambda \in \alpha+ \beta \, \Lambda_{n,\frac{\epsilon}{|\beta|}}(A,B) \setminus \sigma(A,B).
\end{align*}
\end{enumerate}
\end{proof}
The following theorem presents some equivalent definitions for $(n,\epsilon)$-pseudospectra of closed linear operator pencil.
\begin{theorem}\label{eq}
Let $A, B\in \CL(\HI)$, $n\in \mathbb{Z}_{+} \cup \{0\}$, and $\epsilon> 0$. Then the following are equivalent.
\begin{enumerate}[(1)]
\item $z\in \Lambda_{n,\epsilon} (A,B)$.
\item $z\in \sigma(A,B) \cup \left\lbrace \lambda \in \mathbb{C}: \,\left\|(\lambda B-A)^{2^n}u\right\|\leq {\epsilon}^{2^n}, u\in D(A)\cap D(B),\|u\|= 1\right\rbrace .$
\item $z\in \sigma(A,B) \cup \Big\{\lambda\in \mathbb{C}: \left[(\lambda B-A)^{2^n}- E\right]u= 0, \|E\|\leq {\epsilon}^{2^n}, u\in  D(A)\cap D(B), \|u\|= 1 \Big\}.$
\end{enumerate}
\end{theorem}
\begin{proof}
(1) $\Longrightarrow$ (2). If $z\in \Lambda_{n,\e}(A,B)\setminus \sigma(A,B)$, then $\left\|(z B-A)^{-2^n} \right\|^{\frac{1}{2^n}} \geq \e^{-1}$ also there exists $x\in D(A)\cap D(B)$ with $\|x\|= 1$ such that $\left\|(z B-A)^{-2^n}x\right\|^{\frac{1}{2^n}} \geq \e^{-1}$. Define $v= (z B-A)^{-2^n}x$ and $\displaystyle u= \frac{v}{\|v\|}$. Then $v\neq 0$ and
\[
\left\|(zB-A)^{2^n}\frac{v}{\|v\|}\right\|= \frac{\|x\|}{\left\|(z B-A)^{-2^n}x\right\|}\leq \e^{2^n}.
\]
(2) $\Longrightarrow$ (3). Suppose $\left\|(z B-A)^{2^n}u \right\|\leq {\epsilon}^{2^n}$ for some $u \in D(A)\cap D(B)$ with $\|u\|=1 $. Then there exists $\phi \in \HI'$ such that $\|\phi\|= 1$ and $\phi(u)= \|u\|= 1$. Define rank one operator $E: \HI\rightarrow \HI$
\[
 E(x)= \phi(x) (z B-A)^{2^n}u.
\]
Then $\|E\|\leq {\epsilon}^{2^n} $ and $\left[(z B-A)^{2^n}- E\right]u= 0$.\\
(3) $\Longrightarrow$ (1). Suppose $\left[(z B - A) ^{2^n}-E\right]u= 0$ for some $E \in \BL(\HI)$ with $\|E\|\leq {\epsilon}^{2^n}$ and $u\in D(A) \cap D(B)$ with $\|u\|= 1$. For $z \notin \sigma(A,B)$, we have $u= (z B -A)^{-2^n}Eu$, also
\begin{align*}
 1= \|u\|= \left\|(z B-A)^{-2^n}Eu\right\| &\leq \left\|(z B -A)^{-2^n}\right\| {\epsilon}^{2^n}.
 \end{align*}
Thus $z \in \Lambda_{n,\epsilon}(A,B).$ 
\end{proof}
\begin{theorem}\label{rethm}
Let $A, B \in \CL(\HI)$ and $B(\lambda_0 B-A)^{-1}$ is bounded for some $\lambda_0\in \rho(A,B)$. Then
\[
(\lambda B-A)^{-1}= (\lambda_0 B-A)^{-1} \sum_{m=0}^{\infty} \left(\lambda-\lambda_0\right)^m \left[B(\lambda_0 B-A)^{-1}\right]^m,
\]
for every $\lambda$ with $\displaystyle \left|\lambda-\lambda_0 \right|< \frac{1}{\left\|B (\lambda_0 B-A)^{-1}\right\|}$.
\end{theorem}
\begin{proof}
Let $\lambda\in \rho(A,B)$, then
\begin{align*}
(\lambda B-A)^{-1} &= -\left[\left(A- \lambda_0 B\right)- \left(\lambda- \lambda_0 \right)B\right]^{-1}\\
&= -\left(A- \lambda_0 B\right)^{-1}\left[I-\left(\lambda- \lambda_0 \right)B\left(A- \lambda_0 B\right)^{-1} \right]^{-1}\\
&= (\lambda_0 B-A)^{-1}\left[I-\left(\lambda- \lambda_0 \right)B(\lambda_0B-A)^{-1}\right]^{-1}.
\end{align*}
If $\displaystyle\left|\lambda-\lambda_0\right| < \frac{1}{\left\|B (\lambda_0B-A)^{-1}\right\|}$, then
\[
(\lambda B-A)^{-1}= (\lambda_0B-A)^{-1} \sum_{m=0}^{\infty}\left(\lambda-\lambda_0\right)^m \left[B(\lambda_0B-A)^{-1}\right]^m.
\]
\end{proof}
\begin{lemma}\label{ana}
Let $A, B\in \CL(\HI)$ and $B(\lambda B-A)^{-1}$ is bounded for every $\lambda\in \rho(A,B)$. Then the resolvent function of the operator pencil $(A, B)$ is analytic on $\rho(A, B)$.
\end{lemma}
\begin{proof}
Let $\lambda, \mu\in \rho(A,B)$,
\[
(\lambda B-A)^{-1}- (\mu B-A)^{-1}= (\mu-\lambda)(\lambda B-A)^{-1}B(\mu B-A)^{-1}. 
\]
Then
\[
\lim_{\lambda\rightarrow \mu} \frac{(\lambda B-A)^{-1}- (\mu B-A)^{-1}}{\lambda-\mu}= -(\mu B-A)^{-1}B(\mu B-A)^{-1}.
\]
\end{proof}
\begin{theorem}
Let $A, B\in \CL(\HI)$, $n\in \mathbb{Z}_{+} \cup \{0\}$, and $\epsilon> 0$. If $B(\lambda B-A)^{-1}$ is bounded for every $\lambda\in \rho(A,B)$, then 
\begin{enumerate}[(i)]
\item  $\Lambda_{n,\e}(A,B)$ is a closed subset of $\C$.
\item Any bounded component of $\Lambda_{n,\e}(A,B)$ has non empty intersection with $\sigma(A,B)$.
\end{enumerate}
\begin{proof}
\begin{enumerate}[(i)]
\item From Theorem \ref{rethm}, $\sigma(A,B)$ is closed. Define $f: \rho(A,B)\rightarrow [0,\infty)$ by
\[
f(\lambda)= \left\|(\lambda B-A)^{-2^n}\right\|^{\frac{1}{2^n}}.
\]
Then $f$ is continuous and 
\[
\left\{\lambda\in \rho(A,B): \left\|(\lambda B-A)^{-2^n}\right\|^{\frac{1}{2^n}}\geq \e^{-1}\right\}= f^{-1}\left(\left[\e^{-1}, \infty\right)\right),
\]
is closed. Hence $\Lambda_{n,\e}(A,B)$ is closed.
\item Suppose $\Omega$ be a bounded component of $\Lambda_{n,\e}(A,B)$ such that $\Omega\cap \sigma(A,B)= \emptyset$. Then 
\[
\Omega \subseteq \rho(A,B)\cap \left\{\lambda \in \rho(A,B): \left\|(\lambda B-A)^{-2^n} \right\|^{\frac{1}{2^n}} \geq \e^{-1} \right\}.
\]
Let $G= \Omega \cap \left\{\lambda \in \rho(A,B): \left\|(\lambda B-A)^{-2^n} \right\| >\e^{-{2^n}} \right\}$. We claim that $G$ is open. Let $\mu\in G$, since $\mu \in \left\{\lambda \in \rho(A,B): \left\|(\lambda B-A)^{-2^n} \right\| >\e^{-{2^n}} \right\}$, there exists $r_\mu >0$, such that $D(\mu, r_\mu) \subset \left\{\lambda \in \rho(A,B): \left\|(\lambda B-A)^{-2^n} \right\| >\e^{-{2^n}} \right\}.$ Since $ \Omega$ is component, $\mu \in \Omega$ and $D(\mu, r_\mu)$ is connected, $D(\mu, r_\mu)\subset \Omega.$  Hence $D(\mu, r_\mu) \subset G$. Now the map 
\[
\lambda \mapsto \left\|(\lambda B-A)^{-2^n} \right\|,
\]
defined from $G$ to $\C$ is subharmonic and continuous on $\overline{G}$. For any boundary point  $\lambda \in G$, $\left\|(\lambda B-A)^{-2^n} \right\|= \e^{-{2^n}}$. But for $\lambda \in G$, $\left\|(\lambda B-A)^{-2^n} \right\| >\e^{-{2^n}}$. This contradicts the maximum principle.
\end{enumerate}
\end{proof}
\end{theorem}
\begin{theorem}
Let $A, B \in \CL(\HI)$ such that $D(A), D(B), D(A+B), D(AB)$ are dense in $\HI$, $n\in \mathbb{Z}_{+} \cup \{0\}$, and $\epsilon> 0$.
\begin{enumerate}[(i)]
\item If $A$ or $ B$ in $\BL(\HI)$, then $\Lambda_{\e}(A^*,B^*)= \overline{\Lambda_{\e}(A^*,B^*)}$.
\item If $A, B \in \BL(\HI)$, then $\Lambda_{n,\e}(A^*,B^*)= \overline{\Lambda_{n,\e}(A^*,B^*)}$.
\end{enumerate}
\end{theorem}
\begin{proof}
\begin{enumerate}[(i)]
\item Let $A$ or $B$ in $\BL(\HI)$, then $A^*+B^*= \left(A+B\right)^*$; see \cite{sun}. If $\lambda \in \sigma(A,B)$ then $\overline{\lambda}\in \sigma(A^*,B^*)$. Further if $\lambda \in \Lambda_{\e}(A,B)\setminus \sigma(A,B)$, then 
\[
\left\|(\lambda B-A)^{-1} \right\|= \left\|\left( \ \overline{\lambda} B^*-A^*\right)^{-1} \right\| \geq \e^{-1}.
\]
The latest step follows as $(\lambda B-A)^{-1}\in \BL(\HI)$; see \cite{tre1}. Hence $\Lambda_{\e}(A,B) \subseteq \overline{\Lambda_{\e}(A^*,B^*)}$. Since $D(A), D(B)$ are dense in $\HI$, the other inclusion follows.
\item Let $A\in \BL(\HI)$, then $(A+B)^*= A^*+B^*$ and $B^*A^*= \left(AB\right)^*$; see \cite{sun}. If $\lambda\in \sigma(A,B)$ then $\overline{\lambda}\in \sigma(A^*, B^*)$. Further if $\lambda\in \Lambda_{n,\e}(A,B)\setminus \sigma(A,B)$ and $B\in \BL(\HI)$,
\[
\left(\left(\lambda B-A\right)^{-2^n}\right)^*= \left( \left(\left(\lambda B-A\right)^{2^n}\right)^{-1}\right)^*= \left( \left(\left(\lambda B-A\right)^{2^n}\right)^{*}\right)^{-1}= \left(\overline{\lambda}B^*-A^*\right)^{-2^n}.
\]
Thus
\[
 \left\|(\lambda B-A)^{-2^n} \right\|^{\frac{1}{2^n}}= \left\|\left(\left(\lambda B-A\right)^{-2^n}\right)^*\right\|^{\frac{1}{2^n}}= \left\|\left( \ \overline{\lambda} B^*-A^*\right)^{-2^n} \right\|^{\frac{1}{2^n}} \geq \e^{-1}.
\]
Hence $\Lambda_{n,\epsilon}(A,B) \subseteq \overline{\Lambda_{n,\e}(A^*,B^*)}$. Since $D(A), D(B)$ are dense in $\HI$, the other inclusion follows.
\end{enumerate}
\end{proof}
The following theorem characterizes the $(n,\e)$-pseudospectrum of a self-adjoint operator pencil.
\begin{theorem}
Let $A, B\in \CL(\HI)$ be self-adjoint operators with $B$ invertible and $AB=BA$. For $n\in \mathbb{Z}_{+}\cup \{0\}$ and $\epsilon>0$,
\[
\Lambda_{n,\e}(A,B)\subseteq  \sigma(A,B)+ \Delta_{\e \left\|B^{-2^n} \right\|^{\frac{1}{2^n}}} \subseteq \mathbb{R}+\Delta_{\e \left\|B^{-2^n} \right\|^{\frac{1}{2^n}}}.
\]
\end{theorem}
\begin{proof}
Suppose $B$ is invertible and self-adjoint. Then $B^{-1}A$ is self-adjoint and $\sigma(A,B)= \sigma(B^{-1}A)\subseteq \mathbb{R}$; see \cite{kre}. Since $AB= BA$, for $\lambda \notin \sigma(A,B)$, (refer \cite{kato})
\[ 
\left\|(\lambda B-A)^{-2^n} \right\|^{\frac{1}{2^n}} \leq \left\|B^{-2^n} \right\|^{\frac{1}{2^n}} \left\|(\lambda I-B^{-1}A)^{-1} \right\|= \frac{\left\|B^{-2^n} \right\|^{\frac{1}{2^n}}}{\text{dist}\left(\lambda, \sigma(B^{-1}A)\right)}.
\]
Thus
\[
\Lambda_{n,\e}(A,B)\subseteq  \sigma(A,B)+ \Delta_{\e \left\|B^{-2^n} \right\|^{\frac{1}{2^n}}} \subseteq \mathbb{R}+\Delta_{\e \left\|B^{-2^n} \right\|^{\frac{1}{2^n}}}.
\]
\end{proof}
\begin{corollary}\label{coro1}
If $A\in \CL(\HI)$ is self-adjoint, $n\in \mathbb{Z}_{+}\cup \{0\}$, and $\epsilon> 0$. Then
\[
\Lambda_{n,\e}(A)\subseteq \sigma(A)+ \Delta_{\e}\subseteq \mathbb{R}+\Delta_{\e}.
\]
\end{corollary}
\section{Pseudospectra of the $2\times 2$ unbounded block operator matrix pencil}
The transformation of the one-dimensional heat equation into  $2\times 2$ unbounded block operator matrix pencil stimulates the development of this section. The generalized Frobenius-Schur factorization examines the spectral, pseudospectral, and $(n,\e)$-pseudospectral enclosures of the $2\times 2$ unbounded block operator pencil. The following provides the generalized Frobenius-Schur factorization for $2\times 2$ unbounded block operator matrix pencil. For more on Frobenius-Schur factorization of block operator matrices, see \cite{jerbi, tret}. Consider the $2 \times 2$ block operator matrices with the entries as closed operators (unbounded operators) on $\HI$. Let $A_i, B_i\in \CL(\HI)$ ($i= 1,2,3,4$), 
\[
\mathcal{A}= \begin{pmatrix}A_1& A_2\\ A_3&A_4\end{pmatrix}\quad \textnormal{and} \quad \mathcal{B}= \begin{pmatrix}B_1& B_2\\ B_3&B_4\end{pmatrix}.
\]
Then
\begin{align*}
D(\mathcal{A})&= \left(D(A_1) \cap D(A_3) \right) \times \left(D(A_2) \cap D(A_4) \right),\\
D(\mathcal{B})&= \left(D(B_1) \cap D(B_3) \right) \times \left(D(B_2) \cap D(B_4) \right).
\end{align*}
Let $\lambda\in \C$, then 
\[
(\mathcal{A},\mathcal{B})(\lambda)= \lambda\mathcal{B}-\mathcal{A}= \begin{pmatrix}\lambda B_1-A_1& \lambda B_2-A_2\\ \lambda B_3-A_3&\lambda B_4-A_4\end{pmatrix}
\]
is defined on $D(\mathcal{A})\cap D(\mathcal{B})$. If $\lambda\in \rho (A_4,B_4)$, then
\begin{equation}\label{fbs2}
\lambda\mathcal{B}-\mathcal{A}=
\begin{pmatrix}I& F_1(\lambda)\\ &I\end{pmatrix}\begin{pmatrix}\lambda B_1-S_1(\lambda)&\\&\lambda B_4-A_4 \end{pmatrix}\begin{pmatrix}I&\\G_1(\lambda)&I\end{pmatrix},
\end{equation}
where $S_1(\lambda)= \left(\lambda B_2-A_2\right)\left(\lambda B_4-A_4\right)^{-1}\left(\lambda B_3-A_3\right)+A_1$, $F_1(\lambda)= \left(\lambda B_2-A_2\right)\left(\lambda B_4-A_4\right)^{-1}$, and $G_1(\lambda)= \left(\lambda B_4-A_4\right)^{-1}\left(\lambda B_3-A_3\right)$. This factorization is called the generalized Frobenius-Schur factorization along the first complement. If $\lambda \in \rho (A_1,B_1)$, then
\begin{align}\label{fbs1}
\lambda\mathcal{B}-\mathcal{A}= \begin{pmatrix}I& \\ F_2(\lambda) &I\end{pmatrix}\begin{pmatrix}\lambda B_1-A_1& \\ & \lambda B_4-S_2(\lambda)\end{pmatrix}\begin{pmatrix}I& G_2(\lambda)\\ &I\end{pmatrix}, 
\end{align}
where $S_2(\lambda)= \left(\lambda B_3-A_3\right)\left(\lambda B_1-A_1\right)^{-1}\left(\lambda B_2-A_2\right)+A_4$, $F_2(\lambda)= \left(\lambda B_3-A_3\right)\left(\lambda B_1-A_1\right)^{-1}$, and $G_2(\lambda)= \left(\lambda B_1-A_1\right)^{-1}\left(\lambda B_2-A_2\right)$. This factorization is called generalized Frobenius-Schur factorization along the second complement. The following hypotheses are considered for further development.
\begin{enumerate}[{\bf(L1)}]
\item $(A_4,B_4)$ is closed and densely defined, $\rho(A_4,B_4) \neq \emptyset$.
\item $D(A_4)\cap D(B_4)\subseteq D(A_2)\cap D(B_2)$ and $F_1(\lambda)$ is bounded for every $\lambda \in \rho(A_4,B_4)$.
\item $G_1(\lambda)$ is bounded on $D(A_3)\cap D(B_3)$ for every $\lambda \in \rho(A_4,B_4)$.
\item $D(A_1) \cap D(B_1)\cap D(A_3) \cap D(B_3)$ is dense and $\lambda B_1-S_1(\lambda)$ is closed for every $\lambda\in \rho(A_4,B_4)$.
\end{enumerate}
\begin{enumerate}[{\bf(H1)}]
\item $(A_1,B_1)$ is closed and densely defined, $\rho(A_1,B_1) \neq \emptyset$.
\item $D(A_1) \cap D(B_1)\subseteq D(A_3) \cap D(B_3)$ and $F_2(\lambda)$ is bounded for every $\lambda \in \rho(A_1,B_1)$.
\item $G_2(\lambda)$ is bounded on $D(A_2)\cap D(B_2)$ for every $\lambda \in \rho(A_1,B_1)$.
\item $D(A_2) \cap D(B_2) \cap D(A_4) \cap D(B_4)$ is dense and $\lambda B_4-S_2(\lambda)$ is closed for every $\lambda \in \rho(A_1,B_1)$.
\end{enumerate}
\begin{theorem}\label{fbsthm2}
Let the block operator matrix pencil $(\mathcal{A}, \mathcal{B})$ satisfies \textnormal{\textbf{(L1)}}-\textnormal{\textbf{(L4)}}. Then $\lambda \in \sigma(\mathcal{A},\mathcal{B})$ if and only if $\lambda\in \sigma(S_1(\lambda), B_1)$. Further for $\lambda \in \rho(\mathcal{A},\mathcal{B})$,
\[
\left(\lambda\mathcal{B}-\mathcal{A}\right)^{-1} 
= \begin{pmatrix}I&\\-G_1(\lambda)&I\end{pmatrix} \begin{pmatrix}\left(\lambda B_1-S_1(\lambda)\right)^{-1}\\ &\left(\lambda B_4-A_4\right)^{-1}\end{pmatrix} \begin{pmatrix}I&-F_1(\lambda)\\&I\end{pmatrix}.
\]
\end{theorem}
\begin{proof}
Suppose the block operator matrix pencil $(\mathcal{A}, \mathcal{B})$ satisfies \textnormal{\textbf{(L1)}}-\textnormal{\textbf{(L4)}}. Then the generalized Frobenius-Schur factoriztion along the first complement $(\ref{fbs2})$ for the block operator matrix pencil $\lambda\mathcal{B}-\mathcal{A}$ gives $\lambda \in \rho(\mathcal{A},\mathcal{B})$ if and only if $\lambda\in \rho(S_1(\lambda), B_1)$, and
\[
\left(\lambda\mathcal{B}-\mathcal{A}\right)^{-1} 
= \begin{pmatrix}I&\\-G_1(\lambda)&I\end{pmatrix} \begin{pmatrix}\left(\lambda B_1-S_1(\lambda)\right)^{-1}\\ &\left(\lambda B_4-A_4\right)^{-1}\end{pmatrix} \begin{pmatrix}I&-F_1(\lambda)\\&I\end{pmatrix}.
\]
\end{proof}
\begin{theorem}\label{fbsthm1}
Let the block operator matrix pencil $(\mathcal{A}, \mathcal{B})$ satisfies \textnormal{\textbf{(H1)}}-\textnormal{\textbf{(H4)}}. Then $\lambda \in \sigma(\mathcal{A},\mathcal{B})$ if and only if $\lambda\in \sigma(S_2(\lambda), B_4)$. Further for $\lambda \in \rho(\mathcal{A},\mathcal{B})$,
\[
\left(\lambda\mathcal{B}-\mathcal{A}\right)^{-1} 
= \begin{pmatrix}I&-G_2(\lambda)\\ &I\end{pmatrix} \begin{pmatrix} \left(\lambda B_1-A_1\right)^{-1} &\\ & \left(\lambda B_4-S_2(\lambda)\right)^{-1}\end{pmatrix} \begin{pmatrix}I&\\- F_2(\lambda)&I\end{pmatrix}.
\]
\end{theorem}
\begin{proof}
Suppose the block operator matrix pencil $(\mathcal{A}, \mathcal{B})$ satisfies \textnormal{\textbf{(H1)}}-\textnormal{\textbf{(H4)}}. Then the generalized Frobenius-Schur factoriztion along the second complement (\ref{fbs1}) for the block operator matrix pencil $\lambda\mathcal{B}-\mathcal{A}$ gives $\lambda\in \rho(\mathcal{A},\mathcal{B})$ if and only if $\lambda\in \rho(S_2(\lambda), B_4)$, and 
\[
\left(\lambda\mathcal{B}-\mathcal{A}\right)^{-1} 
= \begin{pmatrix}I&-G_2(\lambda)\\ &I\end{pmatrix} \begin{pmatrix} \left(\lambda B_1-A_1\right)^{-1} &\\ & \left(\lambda B_4-S_2(\lambda)\right)^{-1}\end{pmatrix} \begin{pmatrix}I&\\- F_2(\lambda)&I\end{pmatrix}.
\]
\end{proof}
The following theorem identifies the spectral inclusion of the $2\times 2$ unbounded block operator matrix pencil using generalized Frobenius-Schur factorizations.
\begin{theorem}
Let the block operator matrix pencil $(\mathcal{A}, \mathcal{B})$ satisfies \textnormal{\textbf{(L1)}}-\textnormal{\textbf{(L4)}} and \textnormal{\textbf{(H1)}}-\textnormal{\textbf{(H4)}}. Then 
\[
\sigma(\mathcal{A},\mathcal{B}) \subseteq \big\lbrace \lambda \in\C : \min \left\lbrace \left\|M_{S_1}(\lambda)\right\|, \left\|N_{S_1}(\lambda)\right\|,\left\|M_{S_2}(\lambda)\right\|, \left\|N_{S_2}(\lambda)\right\|\right\rbrace \geq 1\big\rbrace,
\]
where $M_{S_1}(\lambda), N_{S_1}(\lambda), M_{S_2}(\lambda), N_{S_2}(\lambda)$ are defined subsequently.
\end{theorem}
\begin{proof}
Suppose $(\mathcal{A}, \mathcal{B})$ satisfies (\textbf{L1})-(\textbf{L4}) and (\textbf{H1})-(\textbf{H4}), then
\begin{align*}
\lambda B_1-S_1(\lambda)&= \left(\lambda B_1-A_1\right)-\left(\lambda B_2-A_2\right)\left(\lambda B_4-A_4\right)^{-1}\left(\lambda B_3-A_3\right),\\
\left(\lambda B_1-S_1(\lambda)\right)\left(\lambda B_1-A_1\right)^{-1} &= I-\left[\left(\lambda B_2-A_2\right)\left(\lambda B_4-A_4\right)^{-1}\left(\lambda B_3-A_3\right)\left(\lambda B_1-A_1\right)^{-1}\right],\\
\left(\lambda B_1-A_1\right)^{-1}\left(\lambda B_1-S_1(\lambda)\right) &= I-\left[\left(\lambda B_1-A_1\right)^{-1}\left(\lambda B_2-A_2\right)\left(\lambda B_4-A_4\right)^{-1}\left(\lambda B_3-A_3\right)\right].
\end{align*}
For $\lambda\in \rho(A_1,B_1) \cap \rho(A_4,B_4)$, define
\[
M_{S_1}(\lambda)= \left(\lambda B_2-A_2\right)\left(\lambda B_4-A_4\right)^{-1}\left(\lambda B_3-A_3\right)\left(\lambda B_1-A_1\right)^{-1}.
\]
If $\left\|M_{S_1}(\lambda)\right\|<1$, then $\lambda \in \rho(S_1(\lambda),B_1)$. For $\lambda\in \rho(A_1,B_1) \cap \rho(A_4,B_4)$, define
\[
N_{S_1}(\lambda)= \left(\lambda B_1-A_1\right)^{-1}\left(\lambda B_2-A_2\right)\left(\lambda B_4-A_4\right)^{-1}\left(\lambda B_3-A_3\right).
\]
If $\left\|N_{S_1}(\lambda)\right\|<1$, then $\lambda \in \rho(S_1(\lambda),B_1)$. From Theorem \ref{fbsthm2}, we obtain the spectral inclusion
\begin{equation}\label{sp1}
\sigma(\mathcal{A},\mathcal{B})\subseteq \left\lbrace \lambda\in \C: \left\|M_{S_1}(\lambda)\right\|\geq 1 \right\rbrace \cap \left\lbrace \lambda\in \C: \left\|N_{S_1}(\lambda)\right\|\geq 1 \right\rbrace.
\end{equation}
Similarly,
\begin{align*}
\lambda B_4-S_2(\lambda)&= \lambda B_4- A_4- \left(\lambda B_3-A_3\right)\left(\lambda B_1-A_1\right)^{-1}\left(\lambda B_2-A_2\right)\\
\left(\lambda B_4-S_2(\lambda)\right) \left(\lambda B_4- A_4\right)^{-1}&= I -\left[\left(\lambda B_3-A_3\right)\left(\lambda B_1-A_1\right)^{-1}\left(\lambda B_2-A_2\right)\left(\lambda B_4-A_4\right)^{-1}\right]\\
\left(\lambda B_4-A_4\right)^{-1} \left(\lambda B_4-S_2(\lambda)\right)&= I-\left[\left(\lambda B_4-A_4\right)^{-1}\left(\lambda B_3-A_3\right)\left(\lambda B_1-A_1\right)^{-1}\left(\lambda B_2-A_2\right)\right]. 
\end{align*}
For $\lambda\in \rho(A_1,B_1) \cap \rho(A_4,B_4)$, define
\[
M_{S_2}(\lambda)= \left(\lambda B_3-A_3\right)\left(\lambda B_1-A_1\right)^{-1}\left(\lambda B_2-A_2\right)\left(\lambda B_4-A_4\right)^{-1}.
\]
If $\left\|M_{S_2}(\lambda)\right\|<1$, then $\lambda \in \rho(S_2(\lambda),B_4)$. For $\lambda\in \rho(A_1,B_1) \cap \rho(A_4,B_4)$, define
\[
N_{S_2}(\lambda)= \left(\lambda B_4-A_4\right)^{-1}\left(\lambda B_3-A_3\right)\left(\lambda B_1-A_1\right)^{-1}\left(\lambda B_2-A_2\right).
\]
If $\left\|N_{S_2}(\lambda)\right\|< 1$, then $\lambda \in \rho(S_2(\lambda),B_4)$. From Theorem \ref{fbsthm1},
\begin{equation}\label{sp2}
\sigma(\mathcal{A},\mathcal{B})\subseteq \left\lbrace \lambda\in \C: \left\|M_{S_2}(\lambda)\right\|\geq 1 \right\rbrace \cap \left\lbrace \lambda\in \C: \left\|N_{S_2}(\lambda)\right\|\geq 1 \right\rbrace.
\end{equation}
Combining the inclusions $(\ref{sp1})$ and $(\ref{sp2})$, we obtain
\[
\sigma(\mathcal{A},\mathcal{B}) \subseteq \big\lbrace \lambda \in\C : \min \left\lbrace \left\|M_{S_1}(\lambda)\right\|, \left\|N_{S_1}(\lambda)\right\|,\left\|M_{S_2}(\lambda)\right\|, \left\|N_{S_2}(\lambda)\right\|\right\rbrace \geq 1\big\rbrace.
\]
\end{proof}
Using the generalized Frobenius-Schur factorizations, the following theorems identify pseudospectral inclusions for the $2\times 2$ unbounded block operator matrix pencil.
\begin{theorem}\label{pseu1}
Let the block operator matrix pencil $(\mathcal{A}, \mathcal{B})$ satisfies \textnormal{\textbf{(H1)}}-\textnormal{\textbf{(H4)}}. Then for $\e>0$,
\[
\Lambda_{\e}(\mathcal{A},\mathcal{B}) \subseteq \left(\Lambda_{\e (1+\delta_1)(1+\delta_2)}\left(A_1, B_1\right)\right) \cup \left(\Lambda_{\e (1+\delta_1)(1+\delta_2)}\left(S_2(\lambda),B_4\right)\right),
\]
where $\delta_1= \left\|\left(\lambda B_1-A_1\right)^{-1}\left(\lambda B_2-A_2\right)\right\|$ and $\delta_2= \left\|\left(\lambda B_3-A_3\right)\left(\lambda B_1-A_1\right)^{-1}\right\|$.
\end{theorem}
\begin{proof}
Suppose $(\mathcal{A}, \mathcal{B})$ satisfies \textbf{(H1)}-\textbf{(H4)} and $\lambda\in \rho(\mathcal{A}, \mathcal{B})$. From Theorem \ref{fbsthm1}, $\lambda \in \rho(A_1,B_1)\cap \rho (S_2(\lambda), B_4)$ and 
\[
\left(\lambda\mathcal{B}-\mathcal{A}\right)^{-1} 
= \begin{pmatrix}I& -G_2(\lambda)\\ &I\end{pmatrix} \begin{pmatrix}\left(\lambda B_1-A_1\right)^{-1} & \\& \left(\lambda B_4-S_2(\lambda)\right)^{-1}\end{pmatrix} \begin{pmatrix}I&\\- F_2(\lambda) &I\end{pmatrix},
\]
where $G_2(\lambda)= \left(\lambda B_1-A_1\right)^{-1}\left(\lambda B_2-A_2\right)$ and $F_2(\lambda)= \left(\lambda B_3-A_3\right)\left(\lambda B_1-A_1\right)^{-1}$. Then
\begin{align*}
\left\|\begin{pmatrix}I& -\left(\lambda B_1-A_1\right)^{-1}\left(\lambda B_2-A_2\right)\\ &I\end{pmatrix}\right\|&\leq  \left\|\begin{pmatrix}I&\\&I\end{pmatrix}\right\|+ \left\|\begin{pmatrix}&&-\left(\lambda B_1-A_1\right)^{-1}\left(\lambda B_2-A_2\right)\\ &&\end{pmatrix}\right\|\\
&\leq 1 +\left\|\left(\lambda B_1-A_1\right)^{-1}\left(\lambda B_2-A_2\right)\right\|,
\end{align*}
and
\begin{align*}
\left\|\begin{pmatrix}I&\\-\left(\lambda B_3-A_3\right)\left(\lambda B_1-A_1\right)^{-1} &I\end{pmatrix}\right\|&\leq  \left\|\begin{pmatrix}I& \\ &I\end{pmatrix}\right\|+ \left\|\begin{pmatrix} &&\\ -\left(\lambda B_3-A_3\right)\left(\lambda B_1-A_1\right)^{-1}&& \end{pmatrix}\right\| \\
&\leq 1 + \left\|\left(\lambda B_3-A_3\right)\left(\lambda B_1-A_1\right)^{-1}\right\|.
\end{align*}
Denote $\delta_1= \left\|\left(\lambda B_1-A_1\right)^{-1}\left(\lambda B_2-A_2\right)\right\|$ and $\delta_2= \left\|\left(\lambda B_3-A_3\right)\left(\lambda B_1-A_1\right)^{-1}\right\|$. Then
\[
\left\|\left(\lambda \mathcal{B}-\mathcal{A}\right)^{-1}\right\| \leq \left(1+\delta_1\right)\left(1+\delta_2\right) \max \left\lbrace\left\|\left(\lambda B_1-A_1\right)^{-1}\right\|, \left\|\left(\lambda B_4-S_2(\lambda)\right)^{-1}\right\| \right\rbrace.
\]
Thus for $\e>0$,
\[
\Lambda_{\e}(\mathcal{A},\mathcal{B}) \subseteq \left(\Lambda_{\e (1+\delta_1)(1+\delta_2)}\left(A_1, B_1\right)\right) \cup \left(\Lambda_{\e (1+\delta_1)(1+\delta_2)}\left(S_2(\lambda),B_4\right)\right).
\]
\end{proof}
\begin{theorem}\label{pseu2}
Let the block operator matrix pencil $(\mathcal{A}, \mathcal{B})$ satisfies \textnormal{\textbf{(L1)}}-\textnormal{\textbf{(L4)}}. Then for $\e>0$,
\[
\Lambda_{\e}(\mathcal{A},\mathcal{B}) \subseteq \left(\Lambda_{\e (1+\eta_1)(1+\eta_2)}\left(A_4, B_4\right)\right) \cup \left(\Lambda_{\e (1+\eta_1)(1+\eta_2)}\left(S_1(\lambda),B_1\right)\right),
\]
where $\eta_1= \left\|\left(\lambda B_4-A_4\right)^{-1}\left(\lambda B_3-A_3\right)\right\|$ and $\eta_2= \left\|\left(\lambda B_2-A_2\right)\left(\lambda B_4-A_4\right)^{-1}\right\|$.
\end{theorem}
\begin{proof}
Suppose $(\mathcal{A}, \mathcal{B})$ satisfies \textbf{(L1)}-\textbf{(L4)} and $\lambda\in \rho(\mathcal{A}, \mathcal{B})$. From Theorem \ref{fbsthm2}, $\lambda \in \rho(A_4,B_4)\cap \rho (S_1(\lambda), B_1)$ and
\[
\left(\lambda\mathcal{B}-\mathcal{A}\right)^{-1} 
= \begin{pmatrix}I&\\-G_1(\lambda)&I\end{pmatrix} \begin{pmatrix}\left(\lambda B_1-S_1(\lambda)\right)^{-1} & \\ & \left(\lambda B_4-A_4\right)^{-1}\end{pmatrix} \begin{pmatrix}I& - F_1(\lambda)\\ &I\end{pmatrix},
\]
where $G_1(\lambda)= \left(\lambda B_4-A_4\right)^{-1}\left(\lambda B_3-A_3\right)$ and $F_1(\lambda)= \left(\lambda B_2-A_2\right)\left(\lambda B_4-A_4\right)^{-1}$. Denote $\eta_1= \left\|\left(\lambda B_4-A_4\right)^{-1}\left(\lambda B_3-A_3\right)\right\|$ and $\eta_2= \left\|\left(\lambda B_2-A_2\right)\left(\lambda B_4-A_4\right)^{-1}\right\|$. Then
\[
\left\|\left(\lambda \mathcal{B}-\mathcal{A}\right)^{-1}\right\| \leq \left(1+\eta_1\right)\left(1+\eta_2\right) \max \left\lbrace\left\|\left(\lambda B_4-A_4\right)^{-1}\right\|, \left\|\left(\lambda B_1-S_1(\lambda)\right)^{-1}\right\| \right\rbrace.
\]
Hence the result follows.
\end{proof}
The following theorems use the generalized Frobenius-Schur factorizations to identify the $(n,\e)$-pseudospectral inclusions of the $2 \times 2$ unbounded block operator matrix pencil.
\begin{theorem}\label{npseu1}
Let the block operator matrix pencil $(\mathcal{A}, \mathcal{B})$ satisfies \textnormal{\textbf{(H1)}}-\textnormal{\textbf{(H4)}}, $n\in \mathbb{Z}_{+}$, and $\e>0$. Further $G_2(\lambda)F_2(\lambda)= F_2(\lambda)G_2(\lambda)= 0$, $G_2(\lambda)(\lambda B_4-S_2(\lambda))^{-1}= (\lambda B_1-A_1)^{-1} G_2(\lambda)$, and $F_2(\lambda)(\lambda B_1-A_1)^{-1}= (\lambda B_4-S_2(\lambda))^{-1} F_2(\lambda)$ for every $\lambda\in \rho(\mathcal{A},\mathcal{B})$. Then
\[
\Lambda_{n,\e}(\mathcal{A},\mathcal{B}) \subseteq \left(\Lambda_{n,\e \left(1+\delta_1^{2^n}\right)^\frac{1}{2^n}\left(1+\delta_2^{2^n}\right)^\frac{1}{2^n}}\left(A_1, B_1\right)\right) \cup \left(\Lambda_{n,\e \left(1+\delta_1^{2^n}\right)^\frac{1}{2^n}\left(1+\delta_2^{2^n}\right)^\frac{1}{2^n}}\left(S_2(\lambda),B_4\right)\right),
\]
where $\delta_1= \left\|\left(\lambda B_1-A_1\right)^{-1}\left(\lambda B_2-A_2\right)\right\|$ and $\delta_2= \left\|\left(\lambda B_3-A_3\right)\left(\lambda B_1-A_1\right)^{-1}\right\|$.
\end{theorem}
\begin{proof}
Suppose $\lambda\in \rho(\mathcal{A},\mathcal{B})$, from Theorem \ref{fbsthm1},
\[
\left(\lambda\mathcal{B}-\mathcal{A}\right)^{-1} 
= \begin{pmatrix}I& -G_2(\lambda)\\ &I\end{pmatrix} \begin{pmatrix}\left(\lambda B_1-A_1\right)^{-1} & \\& \left(\lambda B_4-S_2(\lambda)\right)^{-1}\end{pmatrix} \begin{pmatrix}I&\\- F_2(\lambda) &I\end{pmatrix}.
\]
If $G_2(\lambda)F_2(\lambda)= F_2(\lambda)G_2(\lambda)= 0$, $G_2(\lambda)(\lambda B_4-S_2(\lambda))^{-1}= (\lambda B_1-A_1)^{-1} G_2(\lambda)$, and $F_2(\lambda)(\lambda B_1-A_1)^{-1}= (\lambda B_4-S_2(\lambda))^{-1} F_2(\lambda)$, then the decomposition matrices are mutually commutative. For $n\in \mathbb{Z}_+ $,
\[
\left(\lambda\mathcal{B}-\mathcal{A}\right)^{-2^n} 
= \begin{pmatrix}I& G_2(\lambda)^{2^n}\\ &I\end{pmatrix} \begin{pmatrix}\left(\lambda B_1-A_1\right)^{-2^n} & \\& \left(\lambda B_4-S_2(\lambda)\right)^{-2^n}\end{pmatrix} \begin{pmatrix}I&\\F_2(\lambda)^{2^n} &I\end{pmatrix}.
\]
Also
\begin{align*}
\left\|\begin{pmatrix}I& G_2(\lambda)^{2^n}\\ &I\end{pmatrix}\right\|&=
\left\|\begin{pmatrix}I& \left(\left(\lambda B_1-A_1\right)^{-1}\left(\lambda B_2-A_2\right)\right)^{2^n}\\ &I\end{pmatrix}\right\|\\
&\leq  \left\|\begin{pmatrix}I&\\&I\end{pmatrix}\right\|+ \left\|\begin{pmatrix}&&\left(\left(\lambda B_1-A_1\right)^{-1}\left(\lambda B_2-A_2\right)\right)^{2^n}\\ &&\end{pmatrix}\right\|\\
&\leq 1 +\left\|\left(\left(\lambda B_1-A_1\right)^{-1}\left(\lambda B_2-A_2\right)\right)^{2^n}\right\|.
\end{align*}
Similarly
\begin{align*}
\left\|\begin{pmatrix}I&\\ F_2(\lambda)^{2^n} &I\end{pmatrix}\right\|&= \begin{pmatrix}I& \\\left(\left(\lambda B_3-A_3\right)\left(\lambda B_1-A_1\right)^{-1}\right)^{2^n} &I\end{pmatrix}\\
&\leq  \left\|\begin{pmatrix}I& \\ &I\end{pmatrix}\right\|+ \left\|\begin{pmatrix} &&\\ \left(\left(\lambda B_3-A_3\right)\left(\lambda B_1-A_1\right)^{-1}\right)^{2^n}&& \end{pmatrix}\right\| \\
&\leq 1 + \left\| \left(\left(\lambda B_3-A_3\right)\left(\lambda B_1-A_1\right)^{-1}\right)^{2^n}\right\|.
\end{align*}
Denote $\delta_1= \left\|\left(\lambda B_1-A_1\right)^{-1}\left(\lambda B_2-A_2\right)\right\|$ and $ \delta_2= \left\|\left(\lambda B_3-A_3\right)\left(\lambda B_1-A_1\right)^{-1}\right\|$. Then
\begin{align*}
\left\|\left(\left(\lambda B_1-A_1\right)^{-1}\left(\lambda B_2-A_2\right)\right)^{2^n}\right\|&\leq \left\|\left(\lambda B_1-A_1\right)^{-1}\left(\lambda B_2-A_2\right)\right\|^{2^n}=\delta_1^{2^n},\\
\left\| \left(\left(\lambda B_3-A_3\right)\left(\lambda B_1-A_1\right)^{-1}\right)^{2^n}\right\|&\leq \left\| \left(\lambda B_3-A_3\right)\left(\lambda B_1-A_1\right)^{-1}\right\|^{2^n}=\delta_2^{2^n}.
\end{align*}
Combining the above inequalities,
\[
\left\|\left(\lambda \mathcal{B}-\mathcal{A}\right)^{-2^n}\right\| \leq \left(1+\delta_1^{2^n}\right)\left(1+\delta_2^{2^n}\right) \max \left\lbrace\left\|\left(\lambda B_1-A_1\right)^{-{2^n}}\right\|, \left\|\left(\lambda B_4-S_2(\lambda)\right)^{-{2^n}}\right\| \right\rbrace.
\]
Hence for $n\in \mathbb{Z}_+$ and $\e>0$,
\[
\Lambda_{n,\e}(\mathcal{A},\mathcal{B}) \subseteq \left(\Lambda_{n,\e \left(1+\delta_1^{2^n}\right)^\frac{1}{2^n}\left(1+\delta_2^{2^n}\right)^\frac{1}{2^n}}\left(A_1, B_1\right)\right) \cup \left(\Lambda_{n,\e \left(1+\delta_1^{2^n}\right)^\frac{1}{2^n}\left(1+\delta_2^{2^n}\right)^\frac{1}{2^n}}\left(S_2(\lambda),B_4\right)\right).
\]
\end{proof}
\begin{theorem}\label{npseu2}
Let the block operator matrix pencil $(\mathcal{A}, \mathcal{B})$ satisfies \textnormal{\textbf{(L1)}}-\textnormal{\textbf{(L4)}}, $n\in \mathbb{Z}_+ $, and $\e>0$. Further $G_1(\lambda) F_1(\lambda)= F_1(\lambda) G_1(\lambda)= 0$, $\left(\lambda B_4-A_4\right)^{-1} G_1(\lambda)= G_1(\lambda) \left(\lambda B_1-S_1(\lambda)\right)^{-1}$, and $\left(\lambda B_1-S_1(\lambda)\right)^{-1}F_1(\lambda)= F_1(\lambda)\left(\lambda B_4-A_4\right)^{-1}$ for every $\lambda\in \rho(\mathcal{A}, \mathcal{B})$. Then
\[
\Lambda_{n,\e}(\mathcal{A},\mathcal{B}) \subseteq \left(\Lambda_{n,\e \left(1+\eta_1^{2^n}\right)^\frac{1}{2^n}\left(1+\eta_2^{2^n}\right)^\frac{1}{2^n}}\left(A_4, B_4\right)\right) \cup \left(\Lambda_{n,\e \left(1+\eta_1^{2^n}\right)^\frac{1}{2^n}\left(1+\eta_2^{2^n}\right)^\frac{1}{2^n}}\left(S_1(\lambda),B_1\right)\right),
\]
where $\eta_1= \left\|\left(\lambda B_4-A_4\right)^{-1}\left(\lambda B_3-A_3\right)\right\|$ and $\eta_2= \left\|\left(\lambda B_2-A_2\right)\left(\lambda B_4-A_4\right)^{-1}\right\|$.
\end{theorem}
\begin{proof}
Suppose $\lambda\in \rho(\mathcal{A},\mathcal{B})$, from Theorem \ref{fbsthm2},
\[
\left(\lambda\mathcal{B}-\mathcal{A}\right)^{-1} 
= \begin{pmatrix}I&\\-G_1(\lambda)&I\end{pmatrix} \begin{pmatrix}\left(\lambda B_1-S_1(\lambda)\right)^{-1} & \\ & \left(\lambda B_4-A_4\right)^{-1}\end{pmatrix} \begin{pmatrix}I& - F_1(\lambda)\\ &I\end{pmatrix}.
\]
The conditions assumed in the statement implies the decomposition matrices are mutually commutative. The proof follows similar as Theorem \ref{npseu1}.
\end{proof}
\section{Pseudospectra of the one-dimensional heat operator pencil}
\subsection{Eigenvalue analysis}
Consider the heat transfer through a rod of thermal diffusivity $c^2$ and length $d$. The reformulation of the one-dimensional heat equation subject to the Dirichlet condition at the left and Neumann condition at the right end, along with the variable separable condition, leads to $2\times 2$ block operator matrix pencil $(A, B)$, where
\begin{equation}\label{block}
A= \begin{pmatrix} c^2\cfrac{d^2}{dx^2}&&\\ && \end{pmatrix} \ \ \ \text{and} \ \ \ B= \begin{pmatrix}I&\\ \cfrac{d}{dx} &-I\end{pmatrix}.
\end{equation}
Let $\lambda$ be a generalized eigenvalue and $\psi= \begin{pmatrix}\psi_1\\\psi_2 \end{pmatrix}$ be the corresponding generalized eigenfunction of the heat operator pencil, then the boundary conditions give $\psi_1(0)= \psi_2(d)= 0$. The adjoint of $A$ and $B$ are
\[
A^*= \begin{pmatrix} c^2\cfrac{d^2}{dx^2} &&\\ && \end{pmatrix} \ \ \ \text{and} \ \ \ B^*= \begin{pmatrix}I& \cfrac{d}{dx}\\ &-I\end{pmatrix}.
\]
The domains $D(A^*), D(B^*)$ are dense subspaces of Hilbert space $L^2[0,d]\times L^2[0,d]$ defined by
\begin{align*}
D(A^*)&= \left\{v\in L^2[0,d]: v' \ \textnormal{is absolutely continous}, \ v''\in L^2[0,d]\right\}\times L^2[0,d],\\
D(B^*)&= L^2[0,d]\times \left\{v \in L^2[0,d]: v \ \textnormal{is absolutely continous}, \ v'\in L^2[0,d]\right\}.
\end{align*}
The $2\times 2$ block operator matrix pencil $(A, B)$ corresponding to the one-dimensional heat equation is non-self-adjoint. The generalized eigenvalues and the generalized eigenfunctions of $(A, B)$ are determined as follows. Let $\psi=\begin{pmatrix} \psi_1 \\ \psi_2 \end{pmatrix}$, then
\begin{align*}
(\lambda B- A)\psi= \begin{pmatrix}\lambda I- c^2\cfrac{d^2}{dx^2} & \\ \lambda \cfrac{d}{dx}& -\lambda I\end{pmatrix}  \begin{pmatrix} \psi_1\\ \psi_2
\end{pmatrix}.
\end{align*}
Thus $(\lambda B- A)\psi= 0$ gives two simultaneous differential equations
\begin{align*}
c^2\cfrac{d^2\psi_1}{dx^2}= \lambda \psi_1 \ \ \ \text{and} \ \ \
\lambda \cfrac{d\psi_1}{dx}= \lambda \psi_2.
\end{align*}
By solving the above system of equations, we get
\[
\psi_1(x) = \left(\alpha e^{\frac{\sqrt{\lambda}x}{c}}+\beta e^{-\frac{\sqrt{\lambda}x}{c}}\right) \qquad \textnormal{and} \qquad \psi_2(x) =\frac{\sqrt{\lambda}}{c} \left(\alpha e^{\frac{\sqrt{\lambda}x}{c}}-\beta e^{-\frac{\sqrt{\lambda}x}{c}}\right),
\]
for some arbitary constants $\alpha, \beta$. The boundary conditions are $\psi_1(0)= \psi_2(d)= 0$ and hence $\lambda=0$ can not be an eigenvalue. The boundary condition $\psi_1(0)= 0$ implies
\begin{equation*}\label{egnfn}
\psi(x)=\begin{pmatrix}
\sinh \frac{\sqrt{\lambda}x}{c} \\ \frac{\sqrt{\lambda}}{c}\, \cosh \frac{\sqrt{\lambda}x}{c} 
\end{pmatrix},
\end{equation*} 
where $\lambda$ is the non-zero generalized eigenvalue of the operator pencil $(A,B)$. The condition $\psi_2(d)= 0$ implies $\cosh \frac{\sqrt{\lambda}d}{c}= 0$ and 
\[
 \sqrt{\lambda_n}= c\left(n+\frac{1}{2}\right)\frac{\pi i}{d}.
\]
Hence the generalized eigenvalues of the operator pencil $(A, B)$ are
\begin{equation*}\label{egn}
\lambda_n= -c^2\left(n+\frac{1}{2}\right)^2 \left(\frac{\pi}{d}\right)^2,
\end{equation*}
and the corresponding generalized eigenfunctions are
\[
\psi^n(x)=\begin{pmatrix}
\sinh \frac{\sqrt{\lambda_n}x}{c} \\ \frac{\sqrt{\lambda_n}}{c}\, \cosh \frac{\sqrt{\lambda_n}x}{c}
\end{pmatrix}.
\]
\subsection{Psuedospectral analysis}
This section identifies the pseudospectral enclosure of the heat operator pencil $(A, B)$ using Theorem \ref{pseu1}. Denote $D= \cfrac{d}{dx}$. The generalized Frobenius-Schur factorization along the first complement decomposes the heat operator pencil $(A, B)$ into
\[
\lambda B-A= \begin{pmatrix}
I&\\&I
\end{pmatrix} \begin{pmatrix}
\lambda I-c^2D^2&\\&-\lambda I
\end{pmatrix} \begin{pmatrix}
I&\\-D&I
\end{pmatrix}.
\]
The generalized Frobenius-Schur factorization along the second complement decomposes the heat operator pencil $(A, B)$ into
\[
\lambda B-A= \begin{pmatrix} I&\\ \lambda D\left(\lambda I-c^2D^2\right)^{-1}&   I\end{pmatrix} \begin{pmatrix} 
\lambda I-c^2D^2&\\ &-\lambda I \end{pmatrix} \begin{pmatrix}
I&\\&I
\end{pmatrix}.
\]
\begin{theorem}\label{thm1}
Let $(A, B)$ be the heat operator pencil and $\e>0$. Then
\[
\Lambda_{\e}(A,B)\subseteq c^2\sigma(D^2)+\Delta_{\e(1+\delta_1)},
\]
where $\delta_1= |\lambda|\left\|D\left(\lambda I-c^2D^2 \right)^{-1}\right\|$ and $\Delta_{\e(1+\delta_1)}= \{z\in \C: |z|<\e(1+ \delta_1)\}$.
\end{theorem}
\begin{proof}
For $\lambda\in \mathbb{C}$, 
\[ 
\lambda B-A= \begin{pmatrix} 
\lambda I-c^2\cfrac{d^2}{dx^2}&\\ \lambda \cfrac{d}{dx}&  -\lambda I
\end{pmatrix} = \begin{pmatrix} 
\lambda I-c^2D^2&\\ \lambda D& -\lambda I
\end{pmatrix}
\]
For $\lambda\in \rho(D^2)$, the generalized Frobenius-Schur factorization along the second complement decomposes the heat operator pencil $(A, B)$ into
\[
\lambda B-A= \begin{pmatrix} I&\\ \lambda D\left(\lambda I-c^2D^2\right)^{-1}&   I\end{pmatrix} \begin{pmatrix} 
\lambda I-c^2D^2&\\ &-\lambda I \end{pmatrix} \begin{pmatrix}
I&\\&I
\end{pmatrix}
\] 
From Theorem \ref{fbs1}, for $\lambda\in \rho(D^2)$,
\[
(\lambda B-A)^{-1}= \begin{pmatrix}
I&\\&I
\end{pmatrix}\begin{pmatrix} 
 (\lambda I-c^2D^2)^{-1}&\\ &-\cfrac{1}{\lambda}I \end{pmatrix}\begin{pmatrix} 
I&\\ -\lambda D\left(\lambda I-c^2D^2\right)^{-1}&I\end{pmatrix}.
\] 
From Green's function, for $f\in L^2[0,d]$,
\[
\left(\lambda I-c^2D^2\right)^{-1}(f)= \frac{1}{2\sqrt{\lambda}c} \int_{x}^{d} \left(e^{\frac{\sqrt{\lambda}(x-t)}{c}}-e^{-\frac{\sqrt{\lambda}(x-t)}{c}}\right) \, f(t) \ dt.
\]
Then 
\begin{align*}
D\left(\lambda I-c^2D^2 \right)^{-1}(f)&= \frac{1}{2c^2} \int_{x}^{d} \left(e^{\frac{\sqrt{\lambda}(x-t)}{c}}+e^{-\frac{\sqrt{\lambda}(x-t)}{c}}\right)\, f(t) \ dt\\
&= \frac{1}{2c} \left[\frac{1}{c}\int_x^d e^{\frac{\sqrt{\lambda}(x-t)}{c}} f(t)\, dt + \frac{1}{c}\int_x^d e^{-\frac{\sqrt{\lambda}(x-t)}{c}} f(t)\, dt \right]\\
&= \frac{1}{2c}\left[\left(\sqrt{\lambda}I-cD\right)^{-1}(f)+ \left(\sqrt{\lambda}I+cD\right)^{-1}(f)\right].
\end{align*}
Hence
\[
\lambda D\left(\lambda I- c^2D^2 \right)^{-1}= \frac{\lambda}{2c}\left[ \left(\sqrt{\lambda}I- cD\right)^{-1}+ \left(\sqrt{\lambda}I+ cD\right)^{-1}\right].
\]
From Theorem \ref{pseu1},
\[
\Lambda_{\e}(A,B) \subseteq \Lambda_{\e (1+\delta_1)}\left(c^2D^2\right) \cup \Lambda_{\e(1+ \delta_1)}({\bf 0})= c^2\,\Lambda_{\frac{\e(1+\delta_1)}{c^2}}\left(D^2\right) \cup \Delta_{ \e(1+ \delta_1)},
\]
where $\delta_1= \left\|\lambda D\left(\lambda I-c^2D^2 \right)^{-1}\right\|$ and $\Delta_{\e(1+\delta_1)}= \left\{z\in \C: |z|\leq \e(1+ \delta_1)\right\}$. Since $D^2$ is self-adjoint and $0\in \sigma(D^2)$, from Corollary \ref{coro1},
\[
c^2\,\Lambda_{\frac{\e(1+\delta_1)}{c^2}}\left(D^2\right) \cup \Delta_{ \e(1+ \delta_1)}= c^2\sigma(D^2)+\Delta_{\e(1+\delta_1)}.
\]
\end{proof}
\begin{proposition}\label{prop1}
Let $v= \begin{pmatrix}v_1\\ v_2 \end{pmatrix}\in L^2[0,d]\times L^2[0,d]$. Define
\begin{align*}
w(x)=\begin{cases} \frac{-v_1(-x)+v_2(-x)}{\sqrt{2}}, & -d \leq x\leq 0,\\
\frac{v_1(x)+v_2(x)}{\sqrt{2}},& 0\leq x\leq d.
\end{cases}
\end{align*}
Then $w\in L^2[-d,d]$ and $\|w\|= \|v\|$.
\end{proposition}
\begin{proof}
Suppose $v_1, v_2\in L^2[0,d]$ and $w$ be defined as above, then
\begin{align*}
\|w\|^2 &= \int_{-d}^{d} \left|v(x)\right|^2 dx=  \frac{1}{2}\int_{-d}^{0} \left| -v_1(-x)+v_2(-x)\right|^2 dx + \frac{1}{2} \int^{d}_{0} \left| v_1(x)+v_2(x)\right|^2 dx \\
&= -\frac{1}{2}\int_{d}^{0} \left|-v_1(x)+v_2(x)\right|^2 dx + \frac{1}{2} \int^{d}_{0} \left|v_1(x)+v_2(x)\right|^2 dx\\
& = \frac{1}{2} \int^{d}_{0} \left|-v_1(x)+v_2(x)\right|^2 dx + \frac{1}{2}\int^{d}_{0}\left|v_1(x)+v_2(x)\right|^2 dx\\
& = \int^{d}_{0} \left|v_1(x)\right|^2+  \left|v_2(x)\right|^2 dx =  \|v\|^2.
\end{align*}
\end{proof}
The following provides another reformulation of the one-dimensional heat equation.
\begin{remark}
Consider the heat equation
\begin{align*}
\frac{\partial \phi}{\partial t}= c^2\frac{\partial^2 \phi}{\partial x^2}, \ \ 0\leq x\leq d, \ \  t\geq 0, \ \  \phi(0,t)= \phi_x(d,t)= 0.
\end{align*}
Suppose $\phi$ is variable separable and $\phi(x,t)= \phi_1(x)\,\phi_2(t)$. Define  $\xi\in L^2[-d,d]$ by
\[
\xi(x)=\begin{cases} \frac{-\phi_1(-x)+\phi_1'(-x)}{\sqrt{2}}, & -d \leq x\leq 0,\\
\frac{\phi_1(x)+\phi_1'(x)}{\sqrt{2}},& 0\leq x\leq d.
\end{cases}
\]
From Proposition \ref{prop1}, $\|\xi\|= \|\phi_1\|$. For $\phi_1(x)= \xi(x)$ and $\phi_2(t)= e^{\lambda t}$, the system becomes an eigenvalue problem of the form
\[
(\lambda I-L)\xi= 0, \ \ L= c^2\cfrac{d^2}{dx^2}, \ \ \xi(d)= -\xi(-d).
\]
\end{remark}
\subsection{Discretization}  
Consider a grid $0= x_0< x_1< \cdots < x_m= d$, $x_i= i\,\delta x$ and $0= t_0<t_1<\cdots$, $t_j= t_0+j\,\delta_t$. Denote $\phi_{i}^j= \phi(x_i, t_j)$, the forward time and centered space approximation of the one-dimensional heat equation results in
\[
\phi_{i}^{j+1}= \phi_{i}^{j}+\frac{c^2\,\delta t}{(\delta x)^2}\left(\phi_{i+1}^{j}-2\phi_{i}^{j}+\phi_{i-1}^{j}\right),
\]
where $i= 0,1, \ldots, m$ and $j= 0,1, \ldots$. The boundary conditions implies $\phi_0^{j}=0$ and $\cfrac{\phi_{m}^j- \phi_{m-1}^{j}}{\delta x}=0$ for $j=0,1,\ldots$. The above becomes a linear system
\[
T \phi^{(j)} = r^{(j)},
\] 
where $\phi^{(j)}= \left(\phi_0^j, \ldots, \phi_m^j\right)^T$  and $r^{(j)}= \left(r_0^j, \ldots, r_m^j\right)^T$ where $r_i^j=\phi_i^{j+1}, r_m^j= 0$.  Denote $a=\cfrac{c^2\,\delta t}{(\delta x)^2}$, then
\[
T= \begin{pmatrix} 1&0\\
a&1-2a&a\\
&a&1-2a&a\\
&&\ddots&\ddots&\ddots\\
&&&a&1-2a&a\\
&&&&-1&1
\end{pmatrix}.
\]
The first and last rows are adjusted as per the boundary conditions. We depict the plots of pseudospectrum for $T_{10 \times 10}$ for various values of $a$ and $\e$.
\begin{figure}[h!]
\centering
\begin{subfigure}[b]{0.45\textwidth}
\centering
\includegraphics[width=\textwidth]{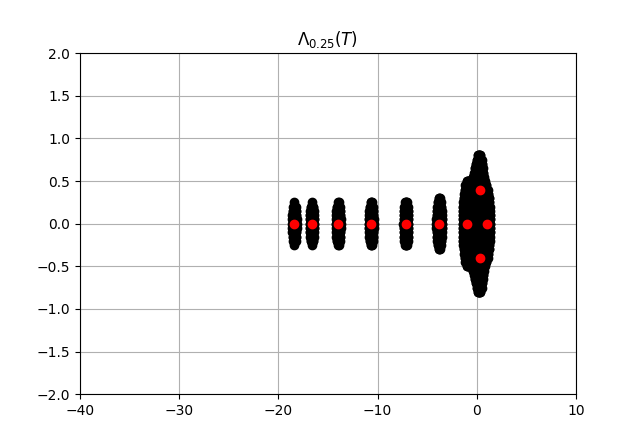}
\caption*{a=5, $\e=0.25$}
\label{fig1}
\end{subfigure}
\hfill
\begin{subfigure}[b]{0.45\textwidth}
\centering
\includegraphics[width=\textwidth]{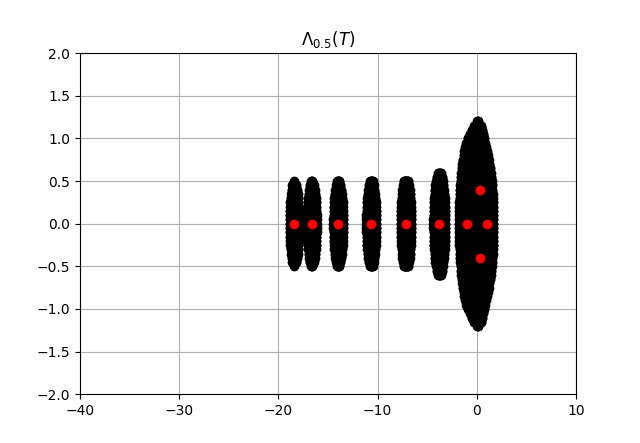}
\caption*{a=5, $\e=0.5$}
\label{fig2}
\end{subfigure}
\end{figure}

\begin{figure}[h!]
\centering
\begin{subfigure}[b]{0.45\textwidth}
\centering
\includegraphics[width=\textwidth]{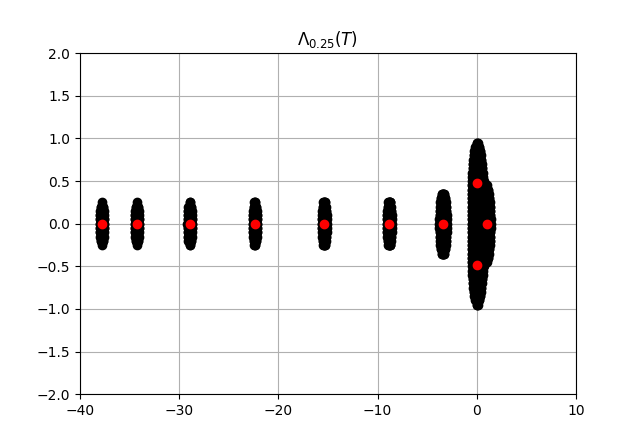}
\caption*{a=10, $\e=0.25$}
\label{fig3}
\end{subfigure}
\hfill
\begin{subfigure}[b]{0.45\textwidth}
\centering
\includegraphics[width=\textwidth]{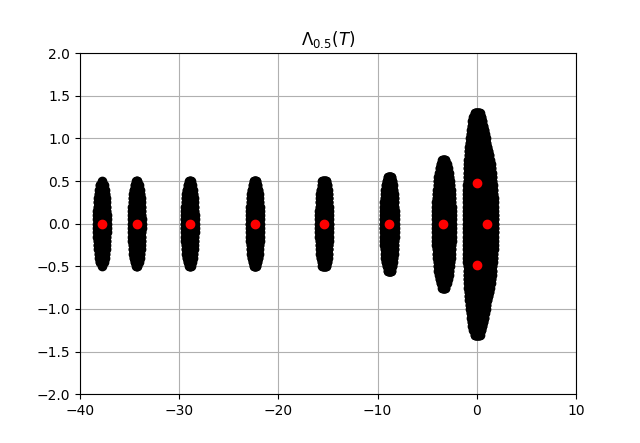}
\caption*{a=10, $\e=0.5$}
\label{fig4}
\end{subfigure}
\end{figure}
The pseudospectra is used to measure the departure from normality or self-adjoint. In Theorem \ref{thm1}, we have seen that the pseudospectra of the heat operator pencil $(A, B)$ is contained in the pseudospectra of the self-adjoint operator $D^2$. The plots of pseudospectra of the discretized one-dimensional heat equation behave in a well-mannered way towards the negative real axis, and we have observed that the eigenvalues away from the origin are less sensitive to perturbations. 
\bibliographystyle{amsplain}

\end{document}